\documentclass[12pt]{amsart}
\usepackage[T1]{fontenc}
\usepackage{txfonts}
\usepackage{times}
\usepackage{amssymb,verbatim}
\usepackage{mathrsfs}
\usepackage[all]{xy}
\usepackage{subfig}
\usepackage{tikz}
\usetikzlibrary{patterns}
\usepackage{color}
\usepackage{a4wide}
\usepackage{hyperref}
\usepackage{wrapfig}
\usepackage{enumerate}
\usetikzlibrary{arrows}
\usepackage{float}

\newcommand{\R}{\mathbb{R}}
\newcommand{\C}{\mathbb{C}}

\newcommand\Z{\mathbb{Z}}

\newcommand{\N}{\mathbb{N}}

\renewcommand{\H}{\mathcal{H}}

\newcommand{\SL}{{\rm SL}}
\newcommand{\GL}{{\rm GL}}

\renewcommand{\S}{\mathbb{S}}
\newcommand{\CP}{\mathbb{CP}}

\newcommand{\Hg}{\mathcal{H}(\underline{k})}

\newcommand{\Hyp}{\mathcal{H}^{\rm hyp}(4)}
\newcommand{\smN}{\mathcal{N}}

\newcommand{\ul}{\underline}

\newcommand{\Id}{\mathrm{Id}}

\newcommand{\Sig}{\Sigma}

\newcommand{\vect}{{\rm Vect}}
\newcommand{\sm}{\setminus}
\newcommand{\sst}{\scriptstyle}

\newcommand{\nc}{\newcommand}
\nc\cC{\mathscr{C}}
\nc\bR{\mathbb{R}}
\nc\bC{\mathbb{C}}
\renewcommand{\bold}[1]{\medskip \noindent {\bf #1 }\nopagebreak}

\newtheorem{Theorem}{Theorem}[section]
\newtheorem{Corollary}[Theorem]{Corollary}
\newtheorem{Lemma}[Theorem]{Lemma}
\newtheorem{Proposition}[Theorem]{Proposition}
\newtheorem{Remark}[Theorem]{Remark}
\newtheorem{Definition}[Theorem]{Definition}

\title{Non-Veech surfaces in $\Hyp$ are generic}
\author{Duc-Manh Nguyen}
\address{\hspace{-0.5cm} IMB Bordeaux-Universit\'e de Bordeaux \newline
351, Cours de la Lib\'eration \newline
33405 Talence Cedex \newline FRANCE}
\email{duc-manh.nguyen@math.u-bordeaux1.fr}

\author{Alex Wright}
\address{\hspace{-0.5cm} University of Chicago\newline
Dept. of Mathematics \newline
5734 S. University Avenue \newline
Chicago, Illinois 60637 \newline USA}
\email{amwright@math.uchicago.edu}
\date{\today}

\begin{document}
\maketitle
\begin{abstract}
We show that every surface in the component $\Hyp$, that is the moduli space of pairs $(M,\omega)$ where $M$ is a genus three hyperelliptic Riemann surface and $\omega$ is an Abelian differential having a single zero on $M$, is either a Veech surface or a generic surface, {\em i.e.} its $\GL^+(2,\R)$-orbit is either a closed or a dense subset of $\Hyp$. The proof develops new techniques applicable in general to the problem of classifying orbit closures, especially in low genus.

Combined with work of Matheus and the second author, a corollary is that there are at most finitely many non-arithmetic Teichm\"uller curves (closed orbits of surfaces not covering the torus) in $\Hyp$.
\end{abstract}

\section{Introduction}

\subsection{General motivation} There is a natural $\GL^+(2,\R)$ action on moduli spaces of translation surfaces, discussed in more detail below. Its study is the central purpose of Teichm\"uller dynamics, and has diverse applications to the dynamics and geometry of interval exchange transformations and rational billiards, and related problems in physics.

Typical questions about translation surfaces include: What are the asymptotics for the number of cylinders of length at most $L$?  How can the surface be decomposed into simpler pieces, respecting the flat geometry? What is the dynamical behavior of the directional flow in most directions? What are the deviations of ergodic averages for the directional flow?

Precisely answering most of these questions for a given surface usually requires first knowing its orbit closure. Indeed, the orbit closure is the arena in which much of the study of a translation surface occurs: The  $\GL^+(2,\R)$ action on this orbit closure provides the renormalization dynamics for the directional flow on the  surface, and how the orbit closure sits inside of the ambient moduli space determines whether and how often certain geometric configurations can be found inside the surface.

\subsection{Background} Let $\ul{k}=(k_1,\dots,k_n)$, with $k_i \in \N$. Recall that $\Hg$ is the moduli space of pairs $(M,\omega)$, where $M$ is a Riemann surface and $\omega$ is a holomorphic one-form (Abelian differential) on $M$ having $n$ zeros of orders $(k_1,\dots,k_n)$. Elements of $\Hg$ are called {\em translation surfaces}. It is well known that $\Hg$ is an algebraic variety and also a complex orbifold of dimension $2g+n-1$ (see \cite{KonZo03, MasTab, Zorich:survey}). \medskip

There is a natural action of $\GL^+(2,\R)$ on each stratum $\Hg$. Every translation surface can be obtained from a collection of polygons in $\R^2$ by gluing pairs of parallel edges with equal length, and the $\GL^+(2,\R)$ action is obtained from the linear action on the polygons in $\R^2$. See the surveys \cite{MasTab,Zorich:survey} for a more detailed introduction. It turns out that the geometric and dynamic features of a specific translation surface are usually encoded in the closure of its $\GL^+(2,\R)$-orbit (see~\cite{Zorich:survey}). \medskip

The space $\Hg$ carries a natural volume form called the {\em Masur-Veech measure}. By the work of Masur and Veech, the action of $\SL(2,\R)$ is ergodic on the locus of unit area surfaces in each connected component of $\Hg$ with respect to this measure. As a consequence, for almost every surface $(M,\omega)$ in $\Hg$, the $\GL^+(2,\R)$-orbit of $(M,\omega)$ is dense in a connected component of $\Hg$. We will call such $(M,\omega)$ {\em generic surfaces}. On the other hand, every stratum contains infinitely many {\em square-tiled surfaces}, and these surfaces have closed  $\GL^+(2,\R)$-orbits.

In genus two, a complete classification of $\GL^+(2,\R)$-orbit closures has been obtained by  McMullen~\cite{McM07, McM05}. There are also some partial results by Calta~\cite{Cal04} and Hubert-Leli\`{e}vre~\cite{HuLe06}.  Using similar ideas, explicit examples of generic surfaces in the hyperelliptic locus $\mathcal{L} \subset \H(2,2)$  and in $\Hyp$  were constructed by Hubert-Lanneau-M\"oller and the first author~\cite{HubLanMo1,HubLanMo2,Ng11}.

\subsection{Recent progress and hopes for the future.}

Recently Eskin-Mirzakhani~\cite{EskMir12}, and Eskin-Mirzakhani-Mohammadi~\cite{EskMirMoh12}, with a contribution by Avila-Eskin-M\"oller~\cite{AvEskMo12}, proved that all $\GL^+(2,\R)$-orbit closures are submanifolds of $\Hg$ which are complex linear subspaces in local period coordinates (see Section~\ref{sec:fund:results}). This confirmed a longstanding conjecture and has lead to further progress. The structure theory of affine invariant manifolds was developed by the second author in \cite{Wri12}, in particular leading to an explicit full measure set of generic translation surfaces in each stratum. Furthermore the geometry of translation surfaces has been directly connected to orbit closures via the Cylinder Deformation Theorem of the second author in \cite{Wri13}; this will be one of our main tools below.

Conjectures of Mirzakhani (see \cite{Wri12})  predict that there are in fact very few orbit closures, and those that do exist have special properties enjoyed by the current list of known examples. This work is the first partial verification of these conjectures in genus  greater than 2.

\subsection{Statement of result.}  Kontsevich-Zorich have classified the connected components of strata \cite{KonZo03}. In particular, there are always at most three connected components. The stratum $\H(4)$ has only two components: $\Hyp$ and $\H^{\rm odd}(4)$. Here $\Hyp$ is the space of pairs $(M,\omega)\in \H(4)$ where $M$ is a hyperelliptic Riemann surface, and $\H^{\rm odd}(4)$ consists of pairs $(M,\omega)$ where $\omega$ defines an odd spin structure on $M$ (see \cite{KonZo03} for a more detailed explanation).

This paper deals with the component $\Hyp$ of the moduli space of translation surfaces of genus three. We show that no proper $\GL^+(2,\R)$-invariant submanifolds of $\Hyp$  exist, other than closed $\GL^+(2,\R)$-orbits.

\begin{Theorem}\label{th:main:Hyp4}
 Let $(M,\omega)$ be a translation surface in $\Hyp$. Then the $\GL^+(2,\R)$-orbit of $(M,\omega)$ is either closed or dense in $\Hyp$.
\end{Theorem}

As a direct consequence, we have the following
\begin{Corollary}\label{cor:algebraic}
Any $\GL^+(2,\R)$ orbit closure in $\Hyp$ is an algebraic variety.
\end{Corollary}

\begin{Remark}
A proof of the conjecture that any  $\GL^+(2,\R)$-orbit closure in any stratum is an algebraic variety has been announced in some recent preprints~\cite{Fil1, Fil2} by S.~Filip.
\end{Remark}

Theorem \ref{th:main:Hyp4} is false for the other connected component $\H^{\rm odd}(4)$ of $\H(4)$. Nonetheless, our methods are applicable to the study of orbit closures in $\H^{\rm odd}(4)$ (see~ \cite{ANW13}) and in even broader generality (for instance in the other strata of three translation surfaces).

Recall the celebrated Veech dichotomy, which states that if the orbit of $(M,\omega)$ is closed, then in every direction with a saddle connection, the surface can be decomposed into parallel cylinders whose moduli are rational multiples of each other; and in every direction without a saddle connection, the straight line flow is uniquely ergodic. A corollary of this dichotomy and Theorem \ref{th:main:Hyp4} is the following.

\begin{Corollary}
If $(M, \omega) \in \Hyp$ has a saddle connection such that in this direction the surface is not the union of parallel cylinders whose moduli are rational multiples of each other, then $(M,\omega)$ is generic.
\end{Corollary}

We suspect that every non-Veech surface in $\Hyp$ has a saddle connection such that in this direction the surface is not the union of parallel cylinders whose moduli are rational multiples of each other.

\subsection{Finiteness of Teichm\"uller curves.} Combining Theorem \ref{th:main:Hyp4} with work in preparation of the second author and Matheus \cite{MaWri}, we  immediately get

\begin{Theorem}[Matheus-Nguyen-Wright]\label{th:fin}
There are at most finitely many non-arithmetic closed orbits (Teichm\"uller curves) in $\Hyp$.
\end{Theorem}
Theorem \ref{th:fin} will be discussed in more detail in \cite{MaWri}, where the appropriate background is already in place. Here we will give only a short discussion of the context.

A closed orbit is non-arithmetic if it does not contain a square-tiled surface. The assumption of non-arithmeticity is weaker than the typical assumption of algebraic primitivity (which will not be defined here). There is a growing literature by Calta and McMullen (independently), McMullen, Bainbridge-M\"oller, M\"oller, and Matheus-Wright (resp.) on the number of algebraically primitive Teichm\"uller curves in $\H(2), \H(1,1),\H(3,1)$,  $\H^{\rm hyp}(g-1,g-1)$ and $\H(2g-2)$ (resp.) \cite{Cal04, McM03, McMuTor, Mo08, BaMo12, MaWri} . All known finiteness results use foundational work of M\"oller \cite{Mo06, Mo06-2}, but the techniques of \cite{MaWri} additionally use Eskin-Mirzakhani-Mohammadi \cite{EskMirMoh12} and  reduce many finiteness problems for Teichm\"uller curves to existence problems for larger orbit closures with unlikely properties.

A weaker statement (finiteness of algebraically primitive Teichm\"uller curves) than Theorem \ref{th:fin}  was conjectured and discussed in the work of Bainbridge and M\"oller \cite{BaMo12}. Bainbridge has informed us that he, M\"oller, and Habegger  have an independent proof of finiteness of algebraically primitive Teichm\"uller curves in $\Hyp$.

Theorem \ref{th:fin} is false in $\H^{\rm odd}(4)$ by work of McMullen \cite{McM06} (see also \cite{Lanneau:Manh:H4}). There are at present two known non-arithmetic Teichm\"uller curves in $\Hyp$, corresponding to the regular $7$-gon and $12$-gon. These examples are due to Veech \cite{Vee89}. The first is algebraically primitive, and the second is not. For a summary of known Teichm\"uller curves, see the notes and references section of \cite{Wri13-2}. Bainbridge and M\"oller  have conjectured that the regular $7$-gon is the only algebraically primitive Teichm\"uller curve in $\Hyp$ \cite[Ex. 14.4]{BaMo12}.

\subsection*{Acknowledgements:}  The first author is grateful to the Universit\'e de Bordeaux, the ANR Project GeoDyM, and the GDR Tresses for their support.  He thanks Erwan Lanneau for helpful comments on this work.

The second author thanks David Aulicino for helpful conversations, and remains always grateful to Alex Eskin and Maryam Mirzakhani for a great many useful conversations on orbit closures.

The authors thank Curt McMullen and the referees for helpful comments on earlier drafts.

\section{Fundamental results}\label{sec:fund:results}
\subsection{Period coordinates}
Let $(M,\omega)$ be an element in a stratum $\Hg$ of translation surfaces of genus $g$.  We denote by $\Sig$ the set of zeros of $\omega$.  The period coordinate mapping is the map $\Phi$ that associates to $(M,\omega)$ the class $[\omega]$ of $\omega$ in the relative cohomology group $H^1(M,\Sig;\C)$. More concretely, pick a set  $\{ c_1,\dots,c_d\}\subset H_1(M,\Sigma;\Z)$ that is a basis of the space $H_1(M,\Sigma;\R)$. We have

$$\Phi(M,\omega) = \left(\int_{c_1}\omega,\dots,\int_{c_d}\omega\right) \in \C^d \simeq H^1(M,\Sig;\C).$$

\noindent Note that $d=2g+n-1$, where $n$ is the cardinality of $\Sig$. If $(M',\omega')$ is an element of $\Hg$ close enough to $(M,\omega)$, then there exists a distinguished isotopy class of homeomorphisms  $f:M \rightarrow M'$ such that $f(\Sig)=\Sig'$, where $\Sig'$ is the set of zeros of $\omega'$. Using $f$ we can identify $H^1(M',\Sig';\C)$ with $H^1(M,\Sig;\C)$. Thus the period coordinate mapping $\Phi$ is well-defined in a neighborhood of $(M,\omega)$ in $\Hg$, with image in $H^1(M,\Sig;\C)$. It is well known that $\Phi$ is a local coordinates for $\Hg$. It follows, in particular that $\dim_\C \Hg=2g+n-1$. For the case of $\H(4)$, we have $g=3,n=1$. Thus $\dim_\C \Hyp=\dim_\C \H(4)=6$.

(Actually we should admit that the isotopy class of $f:M \rightarrow M'$ above may not be well defined if $M$ or $M'$ has automorphisms preserving $\omega$ or $\omega'$. This is a standard and omnipresent issue which has standard solutions. The issue does not occur for $\Hyp$, since surfaces in minimal strata never have automorphisms \cite{MaMoYo}, and so we will not mention it further. An automorphism of a translation surface is a affine self map with trivial derivative.)

%

\subsection{Saddle connections and Cylinders}

Let $(M,\omega)$ be a translation surface. A {\em saddle connection} on $M$ is a geodesic segment, with respect to the flat metric defined by $\omega$, whose endpoints are zeros of $\omega$  (i.e., singularities of the flat metric) and contains no singularities in the interior. Note that the endpoints of a saddle connection need not to be distinct. \medskip

A {\em cylinder} in $M$ is an open subset $C$ of $M$ which is isometric to $\R\times ]0;h[ / \Z,$ where the action of $\Z$ is generated by $(x,y) \mapsto (x+w,y), \; w >0$, and maximal with respect to this property. Here $]0;h[=\{x: 0<x<h\}$ is an open interval in $\R$, and we equip $\R\times ]0,h[$ with the restricted flat metric from $\R^2$. Generally, we will construct cylinders from rectangles (or parallelograms) in $\R^2$ by identifying two opposite sides. We will call $h$ the height, and $w$ the width or circumference of $C$. The {\em modulus} of $C$ is defined to be $h/w$. We will call any simple closed geodesic in $C$ a core curve of $C$. \medskip

Let $f:\R\times]0;h[/\Z \rightarrow M$ be the isometric embedding of $C$. We can extend $f$ by continuity to a map $\bar{f}: \R\times[0;h] \rightarrow M$. We will call the sets $\bar{f}(\R\times\{0\})$ and $\bar{f}(\R\times\{h\})$ the boundary components of $C$. Each boundary component of $C$ is a concatenation of saddle connections, and is freely homotopic to the core curves of $C$. As subsets of $M$, the boundary components of $C$ need not to be disjoint. If each boundary component of $C$ contains only one saddle connection, we call $C$ a {\em simple cylinder}. \medskip

Given any direction $\theta \in \S^1$, we have a well defined unit vector field on $M\sm \Sig$, induced from the constant unit vector field on $\bR^2$ in direction $\theta$. The flow generated by this vector field gives rise to a foliation of $M$, which will be denoted by $\mathcal{F}_\theta$. A direction  is said to be {\em periodic} for $(M,\omega)$, if every leaf of the foliation $\mathcal{F}_\theta$ is either a closed  geodesic for the flat metric, or a saddle connection. If $\theta$ is a periodic direction, then the complement of the union of all saddle connections that are leaves of $\mathcal{F}_\theta$ is a disjoint union of finitely many cylinders. In this case, we also say that $M$ admits a {\em cylinder decomposition} in direction $\theta$. \medskip

\subsection{Affine invariant submanifolds}

\begin{Definition}[Invariant submanifold~\cite{EskMir12, EskMirMoh12}]\label{def:inv:sub:man}
 An  affine invariant submanifold $\smN$ of $\Hg$ is an immersed manifold of $\Hg$, {\em i.e.} $\smN$ is a manifold and there exists a proper immersion\footnote{We are not aware of any example in which this immersion is not an embedding. (It does not seem out of the question to construct such an example using a covering construction, but we have not tried.) The work of \cite{EskMir12, EskMirMoh12} does not rule out the possibility that the immersion is not an embedding, and allows for this possibility in the definition of affine invariant submanifold. Note that this issue is distinct from the possibility that an embedded affine invariant manifold may become immersed when projected to the moduli space of Riemann surfaces.} $f: \smN \rightarrow \Hg$,  such that each point of $\smN$ has a neighborhood whose image is equal to the set of points satisfying a set of real linear equations in  period coordinates. For notational simplicity, generally we will treat affine invariant submanifold as subsets of strata, referring to the image of the immersion.
\end{Definition}


\begin{Theorem}[Eskin-Mirzakhani-Mohammadi]\label{th:ob:cl}
Let $(M,\omega)$ be an element of $\Hg$. Then, the orbit closure $\overline{\GL^+(2,\R)\cdot(M,\omega)}$ is an affine invariant submanifold of $\Hg$.
\end{Theorem}

Let $\smN$ be  an affine invariant submanifold of $\Hg$. Then by definition, in a local chart defined by the period mapping, each fiber of the tangent bundle of $\smN$ can be identified with a fixed vector subspace $T(\smN)$ of $H^1(M,\Sigma;\C)$ which can be written as

$$T(\smN)=\C\otimes_\R T_\R(\smN),$$

\noindent with $T_\R(\smN) \subset H^1(M,\Sigma;\R)$. Let $p: H^1(M,\Sigma;\R) \rightarrow H^1(M,\R)$ be the natural projection. (Note that in the case of minimal strata, and $\Hyp$ in particular,  $p$ is an isomorphism.)

\begin{Theorem}[Avila-Eskin-M\"oller \cite{AvEskMo12}]\label{th:symp:cond}
 Any affine invariant submanifold $\smN$ is symplectic, in the sense that the intersection form is non-degenerate on $p(T_\R(\smN))$.
\end{Theorem}

Theorem~\ref{th:symp:cond} will play a seemingly central role in this paper; however its use could probably be omitted at the expense of some additional analysis.

\section{General cylinder deformation results}

\subsection{Cylinder deformations}  The horocycle flow is defined as part of the $\GL^+(2,\bR)$--action,
$$u_t=\left(\begin{array}{cc} 1&t\\0&1\end{array}\right)\subset \GL^+(2,\bR).$$
We also call the action of $u_t$ the {\em horizontal shear}. We will also be interested in the {\em vertical stretch},
$$a_t=\left(\begin{array}{cc}1&0\\0& e^t\end{array}\right)\subset \GL^+(2,\bR).$$

Given a collection $\cC$ of parallel cylinders on a translation surface $(M,\omega)$, we define the {\em cylinder shear} $u_t^\cC (M,\omega)$ to be the translation surface obtained by shearing the cylinders in $\cC$. This deformation of $(M,\omega)$ can be understood very concretely as follows. Express $(M,\omega)$ as a collection of polygons, including a rectangle for each cylinder in $\cC$, with parallel edge identifications. Since the cylinders of $\cC$ are parallel, their core curves  all make a fixed counterclockwise angle $\theta\in [0,2\pi)$ to the positive horizontal direction in the plane. Define the counterclockwise rotation by $\theta$
$$r_\theta=\left(\begin{array}{cc} \cos \theta &-\sin \theta \\\sin \theta &\cos \theta \end{array}\right).$$

Now $u_t^\cC (M,\omega)$ can be obtained by letting the matrix $r_\theta u_t r_{-\theta}$ act linearly on the rectangles which give cylinders in $\cC$ but not on the remaining polygons, and then regluing.  Similarly we define $a_t^\cC (M,\omega)$ by applying $r_\theta a_t r_{-\theta}$ only to the cylinders in $\cC$, and we call $a_t^\cC (M,\omega)$ the \emph{cylinder stretch}. Both cylinder shear and stretch depend on the  set $\cC$ of parallel cylinders.

Consider a cylinder $C$ on a translation surface $(M,\omega)$. On all sufficiently small deformations of $(M,\omega)$ (i.e., translation surfaces sufficiently close to $(M,\omega)$ in the stratum) there is a corresponding cylinder, whose area and slope might be slightly different.

Now let $\smN$ be an affine invariant submanifold. We say that two cylinders on $(M,\omega)\in \smN$ are $\smN$-parallel  if they are parallel, and remain parallel on deformations of $M$ in $\smN$ \cite{Wri13}. Two cylinders $C_1$ and $C_2$ are $\smN$-parallel if and only if one of the linear equations defining $\smN$ locally in period coordinates at $(M,\omega)$ gives that the holonomy of the core curve of $C_1$ is a fixed constant multiple of the holonomy of the core curve of $C_2$.

For example, if $\smN$ is a closed orbit, then any pair of parallel cylinders on any $(M,\omega)\in \smN$ are $\smN$-parallel. In contrast, if $\smN$ is a connected component of a stratum, then any pair of parallel cylinders on any $(M,\omega)\in \smN$ are $\smN$-parallel if and only if their core curves are homologous. Thus  parallelism is in general far weaker than $\smN$-parallelism, and it is the stability of the parallelism under deformations in $\smN$ that is most important in the definition of $\smN$-parallel.

As another example, fix $d>1$ and some stratum $\Hg$, and let $\smN$ be a connected component of the space of degree $d$ translation covers of surfaces in $\Hg$. For each surface $(M,\omega)\in \smN$, there is a covering map $p$ to a surface surface $(M', \omega')\in \Hg$. Two cylinders on $(M,\omega)$ are $\smN$-parallel if and only if their core curves map to homologous classes under $p_*$.

From \cite{Wri13} we have the following

\begin{Theorem}[The Cylinder Deformation Theorem, Wright]\label{T:CD}
Let $ (M,\omega) \in \smN$ be a translation surface, and let $\cC$ be a collection of parallel cylinders on $(M,\omega)$ so that no cylinder in $\cC$ is $\smN$-parallel to a cylinder on $(M,\omega)$ not in $\cC$. Then for all $s, t\in \bR$, the surface $a_s^\cC(u_t^\cC (M,\omega))$ remains in $\smN$.
\end{Theorem}

 An important special case is when a single cylinder is not $\smN$-parallel to any other. In this case we will say that this cylinder is {\em free}, and the theorem indeed says we are free to stretch and shear this cylinder while remaining in $\smN$.

The relation of cylinders being $\smN$-parallel is an equivalence relation. An equivalence class of $\smN$-parallel cylinders is a maximal collection of pairwise $\smN$-parallel cylinders.

If $C$ is a cylinder and $\cC$ is a collection of cylinders, define $P(C, \cC)$ to be the area of $C\cap (\cup \cC)$ divided by the area of $C$. This is the portion of $C$ which lies in the union of the cylinders in the collection $\cC$.

\begin{Proposition}\label{P:break}
Let $\cC$ and $\cC'$ be equivalence classes of $\smN$-parallel cylinders in different directions. If $C_1, C_2 \in \cC$ then $P(C_1, \cC')=P(C_2, \cC')$.
\end{Proposition}

\begin{proof}
For visualization it might help to assume that the cylinders in $\cC$ are horizontal, and those in $\cC'$ are vertical. (This can be arranged with the $\SL(2,\bR)$-action.)

Let $M_t$ be the result of stretching the cylinders of $\cC'$ by a factor of $t$, so that $M_1=M$. This is defined for $0<t<\infty$. Theorem \ref{T:CD} gives that $M_t\in \smN$.

 Let $c_i(t)$ be the circumference of $C_i$ on $M_t$. Compute
 $$c_i(t) = (1-P(C_i, \cC')+t P(C_i, \cC')) c_i(1).$$
 Thus $c_1(t)/c_2(t)$ is a constant times
 $$\frac{1-P(C_1, \cC')+t P(C_1, \cC')}{1-P(C_2, \cC')+t P(C_2, \cC')}.$$
 Since $C_1$ and $C_2$ are $\smN$-parallel, this must be constant in $t$, and hence the result follows.
\end{proof}

There are some special cases of Proposition \ref{P:break} that are especially simple and useful.

\begin{Proposition}\label{P:morefree}
Let $C'$ be free, and let $C$ be a cylinder that is not parallel to $C'$.
\begin{enumerate}[(a)]
\item If the closure of $C$ contains $C'$, then $C$ is free.
\item If $C$ is contained entirely in the closure of $C'$, and there is no other cylinder parallel to $C$ with this property, then $C$ is free.
\item If $C$  intersects, but is not contained in $C'$, and there is no other cylinder parallel to $C$ with this property, then $C$ is free.
\item If $C$  is disjoint from $C'$, and there is no other cylinder parallel to $C$ with this property, then $C$ is free.
\end{enumerate}
\end{Proposition}

\subsection{Finding many cylinders} Let $\smN$ be an affine invariant submanifold, and let $(M,\omega)$  be a point of $\smN$. Recall that $T_{(M,\omega)}\smN$ is a vector subspace of $H^1(M, \Sigma, \C)$, and that $p(T_{(M,\omega)}\smN)\subset H^1(M,\C)  = H_1(M, \C)^*$. Given any closed curve $\alpha$ in $M$, we consider $\alpha$ as an element of $p(T_{(M,\omega)}\smN)^*$. If $\alpha_1, \ldots, \alpha_n$ are a collection of closed curves of a horizontally periodic translation surface $(M,\omega)$, denote by $\vect(\alpha_1,\ldots,\alpha_n)$  the subspace of $p(T_{(M,\omega)}\smN)^*$ generated by $\alpha_1,\ldots, \alpha_n$. The following is a summary of some relevant results in Section 8 of \cite{Wri13}.

\begin{Theorem}[Wright]\label{T:manyC}
Assume $\smN$ is an affine invariant submanifold. Let $k=\frac12 \dim_\C p(T(\smN))$. Then there is an  $(M,\omega)\in \smN$ that is horizontally periodic with horizontal core curves $\alpha_1, \ldots, \alpha_n$ and $\dim_\C \vect(\alpha_1, \ldots, \alpha_n)=k$.

$\vect(\alpha_1,\ldots,\alpha_n)$ is always an isotropic subspace of the $2k$-dimensional symplectic vector space $p(T_{(M,\omega)}\smN)^*$. Consequently $\dim_\C \vect(\alpha_1, \ldots, \alpha_n)\leq k$.

If the inequality is strict, then it is possible to find another horizontally periodic surface $(M',\omega')$ in $\smN$ with more horizontal cylinders than $(M,\omega)$. Furthermore, if $(M,\omega)$ has a pair of $\smN$-parallel horizontal cylinders, then so does $(M',\omega')$.
\end{Theorem}

In fact more can be said about the passage from $(M,\omega)$ to $(M',\omega))$: all the cylinders on $(M,\omega)$ in a certain sense persist on $(M',\omega')$. See \cite[Section 8]{Wri13} for more details.

A key step in the proofs of the results in this section (the proofs are found in \cite{Wri13}) is the following result from \cite{SmiWei04}

\begin{Theorem}[Smillie-Weiss]\label{T:SW}
Every horocycle flow orbit closure contains a horizontally periodic translation surface.
\end{Theorem}

Note that if  $(M,\omega)\in \smN$  is horizontally periodic with horizontal core curves $\alpha_1, \ldots, \alpha_n$, and two of the horizontal cylinders are $\smN$-parallel, then  $\dim_\C \vect(\alpha_1, \ldots, \alpha_n)\leq n-1$. (This is because the $i$ and $j$ cylinders are $\smN$-parallel if and only if $\dim_\C \vect(\alpha_i, \alpha_j)=1$.) This observation leads to the following result.

\begin{Proposition}\label{P:allfree}
Assume $\smN$ is an affine invariant submanifold. Let $k=\frac12 \dim_\C p(T(\smN))$. Assume that every horizontally periodic translation surface has at most $k$ horizontal cylinders. Then every cylinder on every translation surface in $\smN$ is free.
\end{Proposition}

\begin{proof}
Let $(M,\omega)\in \smN$, and let $C$ be a cylinder on $(M,\omega)$, which without loss of generality is horizontal. Apply Theorem \ref{T:SW} to obtain a horizontally periodic translation surface $(M',\omega')$, which has a horizontal cylinder corresponding to each horizontal cylinder on $(M,\omega)$ (with the same heights). Note that if two cylinders in $(M,\omega)$ are $\smN$-parallel, then so are the corresponding cylinders on $(M',\omega')$.

Now apply Theorem \ref{T:manyC} repeatedly to $(M',\omega')$, to reach another horizontally periodic $(M'',\omega'')$ where the core curves of the horizontal cylinders $\alpha_1, \ldots, \alpha_k$ satisfy $\dim_\bC \vect\{\alpha_1, \ldots, \alpha_k\}=k$ (here we use the assumption that any horizontally periodic surface in $\smN$ has at most $k$ horizontal cylinders).

If $C$ is $\smN$-parallel to some other cylinder on $(M,\omega)$, then there exists a pair of horizontal cylinders on $(M'',\omega'')$ which are  $\smN$-parallel. If the core curves of these two cylinders are $\alpha_i$ and $\alpha_j$, then $\vect\{\alpha_i, \alpha_j\}$ is one dimensional, which contradicts the fact that $\vect\{\alpha_1, \ldots, \alpha_k\}$ is $k$ dimensional.
\end{proof}

\section{Cylinder decompositions in $\Hyp$}

A  computer aided classification of topological models of cylinder decompositions in $\H(4)$ was carried out some time ago by S. Leli\`evre and recently published as an appendix to \cite{MaMoYo}. Also Lindsey \cite{Lin} has relevant results on decompositions of hyperelliptic surfaces.

Since this is central to our analysis, we have included here a self-contained, computer-free, classification of the topological models of cylinder decompositions in $\Hyp$. The reader willing to trust the classification is urged to skip directly to the next section.

We will use the following well-known fact about hyperelliptic translation surfaces (see \cite{KonZo03}, \cite{Ng11}). We include its proof here for the reader's convenience.

\begin{Lemma}\label{lm:hyp:cyl:inv}
 Let $(M,\omega)$ be a translation surface in one of the components $\H(2g-2)^{\rm hyp}$ or $\H(g-1,g-1)^{\rm hyp}$. If $C$ is a cylinder on $M$, then $C$ is preserved by the hyperelliptic involution.
\end{Lemma}
\begin{proof}
Let $\pi: M \rightarrow \CP^1$ be the double covering which is ramified at $2g+2$ points of $M$. Let $\eta$ be the quadratic differential on $\CP^1$ such that $\pi^*\eta = \omega^2$. Since $(M,\omega) \in \H(2g-2)^{\rm hyp}\sqcup \H(g-1,g-1)^{\rm hyp}$, $\eta$  has a unique zero. This is because in  hyperelliptic strata, the hyperelliptic involution either fixes the unique zero of order $2g-2$, or exchanges the two zeros of order $g - 1$ ($\H(2g-2)^{\rm hyp}$ and $\H(g-1,g-1)^{\rm hyp}$ can be respectively identified with $\mathcal{Q}(-1^{2g+1},2g-3)$  and $\mathcal{Q}(-1^{2g+2},2g-2)$, see \cite{KonZo03}, Definition 2). Let $N$ denote the flat surface with conical singularities defined by $\eta$.

Let $\rho$ denote the hyperelliptic involution of $M$, and $c$ be a  simple closed geodesic contained in $C$. Note that since $\rho^*\omega=-\omega$, $\rho$ is an isometry of the flat metric on $M$ whose differential is $-\Id$. In particular $\rho(c)$ is a simple closed geodesic parallel to $c$ (but with opposite orientation). Thus either $c=\rho(c)$ as subsets of $M$, or $c$ and $\rho(c)$ are disjoint. Clearly, we only need to consider the latter case, which means that the projection of $c$ on $\CP^1$ is a simple closed curve.

Consider $\hat{c}=\pi(c) \subset N$. By definition, $\hat{c}$ is a simple closed geodesic in $N$. Since $N$ is homeomorphic to the sphere, $\hat{c}$ cuts $N$ into two disks. Let $D_0$ denote the one that does not contain the unique zero of $\eta$. By the Gauss-Bonnet Theorem, $D_0$ must contain some singularities of the flat metric. By assumption, these singularities are poles of $\eta$ which are images of the branched points of $\pi$. It follows that $C_0=\pi^{-1}(D_0)$ is a connected subsurface of $M$.  Now, $C_0$ is bounded by $c$ and $\rho(c)$ and contains no singularities of the flat metric on $M$. Thus $C_0$ must be a cylinder. By the definition of cylinder on $M$, we can then conclude that $C_0$ is included in $C$ and the lemma follows.
\end{proof}

In what follows, we will say that two parallel cylinders in some translation surface are {\em adjacent} if their boundaries share a common saddle connection. As  direct consequences of  Lemma~\ref{lm:hyp:cyl:inv}, we have

\begin{Corollary}\label{cor:cyl:in:hyp}
Let $(M,\omega)$ be a surface in some hyperelliptic component, which is horizontally periodic. Let $C$ be a horizontal cylinder on $M$.
 \begin{itemize}
  \item[(a)] Let $s$ be a saddle connection invariant by the hyperelliptic involution. If $s$ is contained in the top boundary of $C$, then it is also contained in the bottom boundary of $C$. As a consequence, $s$ cannot be contained in the boundary of another horizontal cylinder.
  \item[(b)] The top and bottom boundary components of $C$ contain the same number of saddle connections. Moreover, the hyperelliptic involution takes the saddle connections on the top boundary of $C$ to those on the bottom boundary, reversing their cyclic order.
  \item[(c)] If $C$ is a simple cylinder, then $C$ is adjacent to a unique horizontal cylinder $C'$.
 \end{itemize}
\end{Corollary}

\begin{proof}
 (a) and (b) follow from the fact that the hyperelliptic involution exchanges the top and bottom boundary components of $C$.

Recall that $C$ is simple if each of its boundary components contains only one saddle connection. If the saddle connection in the top boundary of $C$ belongs to the bottom boundary of some cylinder $C'$, then the one in the bottom boundary of $C$ belongs to the top boundary of $C'$, from which we deduce (c).
\end{proof}

Note that any $(M,\omega) \in \Hyp$ has at most $5$ saddle connections in any fixed direction, and the maximal number is realized if and only if $M$ is periodic in that direction.

\begin{Lemma}\label{lm:3cyl:models}
Let $(M,\omega)$ be a surface in $\H(4)^{\rm hyp}$ for which the horizontal direction is periodic. We denote the horizontal saddle connections of $M$ by $\{1,\dots,5\}$. Then
\begin{itemize}
 \item[(i)] $M$ has at most three horizontal cylinders.
 \item[(ii)] If $M$ has three horizontal cylinders, then the cylinders are glued together following one of the two models in Figure~\ref{fig:3cyl:decomp:models}.

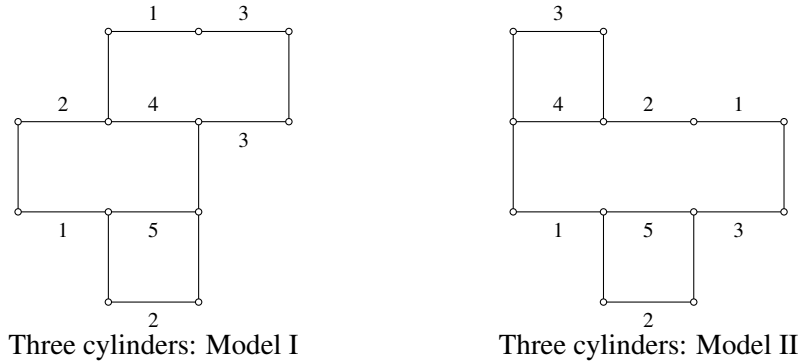
\begin{figure}[htb]
\begin{minipage}[t]{0.4\linewidth}
\centering
\begin{tikzpicture}[scale=0.6]
\draw (0,4) -- (0,2) -- (2,2) -- (2,0) -- (4,0) -- (4,4) -- (6,4) -- (6,6) -- (2,6) -- (2,4) -- cycle;
\draw (2,4) -- (4,4) (2,2) -- (4,2);

\foreach \x in {(0,4), (0,2), (2,6), (2,4), (2,2), (2,0), (4,6), (4,4), (4,2), (4,0), (6,6), (6,4)} \filldraw[fill=white] \x circle (2pt);

\draw (3,6) node[above] {$\sst 1$} (1,2) node[below] {$\sst 1$}  (1,4) node[above] {$\sst 2$}  (3,0) node[below] {$\sst 2$}(5,6) node[above] {$\sst 3$} (5,4) node[below] {$\sst 3$} (3,4) node[above] {$\sst 4$} (3,2) node[below] {$\sst 5$};
\draw (3,-1) node {{\rm\small  Three cylinders: Model I}};

\end{tikzpicture}
\end{minipage}
\begin{minipage}[t]{0.4\linewidth}
\centering
\begin{tikzpicture}[scale=0.6]
\draw (0,6) --  (0,2) -- (2,2) -- (2,0) -- (4,0) -- ( 4,2) -- (6,2) -- (6,4) -- (2,4) --(2,6) -- cycle;
\draw (0,4) -- (2,4) (2,2) -- (4,2);

\foreach \x in {(0,6), (0,4), (0,2) ,(2,6), (2,4), (2,2), (2,0), (4,4), (4,2), (4,0), (6,4), (6,2)} \filldraw[fill=white] \x circle (2pt);

\draw (1,6) node[above] {$\sst 3$} (5,2) node[below] {$\sst 3$} (1,4) node[above] {$\sst 4$} (3,4) node[above] {$\sst 2$} (3,0) node[below] {$\sst 2$} (5,4) node[above] {$\sst 1$} (1,2) node[below] {$\sst 1$} (3,2) node[below] {$\sst 5$};

\draw (3,-1) node{{\rm\small Three cylinders: Model II}};

\end{tikzpicture}
\end{minipage}
\caption{Decomposition into three cylinders of surfaces in $\Hyp$}
\label{fig:3cyl:decomp:models}
\end{figure}

\item[(iii)] If $M$ has two horizontal cylinders, then the cylinders are glued together following one of the two models in Figure~\ref{fig:2cyl:decomp:models}.

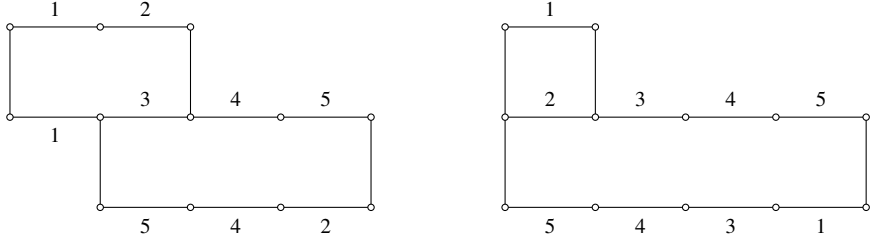
\begin{figure}[htb]
\begin{minipage}[t]{0.4\linewidth}
\centering
\begin{tikzpicture}[scale=0.6]

 \draw (0,2) -- (2,2) -- (2,0) -- (8,0) -- (8,2) -- (4,2) -- (4,4) -- (4,4) -- (0,4) -- cycle;
 \draw ((2,2) -- (4,2);

\foreach \x in {(0,4), (0,2), (2,4), (2,2), (2,0), (4,4), (4,2), (4,0), (6,2), (6,0), (8,2), (8,0)} \filldraw[fill=white] \x circle (2pt);
\draw (1,4) node[above] {$\sst 1$} (1,2) node[below] {$\sst 1$} (3,4) node[above] {$\sst 2$} (3,0) node[below] {$\sst 5$} ((3,2) node[above] {$\sst 3$} (5,2) node[above] {$\sst  4$} (7,0) node[below] {$\sst 2$} (7,2) node[above] {$\sst 5$} (5,0) node[below] {$\sst 4$};
\end{tikzpicture}
\end{minipage}
\begin{minipage}[t]{0.4\linewidth}
 \centering
\begin{tikzpicture}[scale=0.6]
\draw (0,4) -- (0,0) -- (8,0) -- (8,2) -- (2,2)  -- (2,4) -- cycle;
\draw ((0,2) -- (2,2);

\foreach \x in {(0,4),(0,2), (0,0), (2,4), (2,2), (2,0), (4,2), (4,0), (6,2), (6,0), (8,2), (8,0)} \filldraw[fill=white] \x circle (2pt);

\draw (1,4) node[above] {$\sst 1$} (1,0) node[below] {$\sst 5$} (1,2) node[above] {$\sst 2$} (3,2) node[above] {$\sst 3$} (7,0) node[below] {$\sst 1$} (5,2) node[above] {$\sst 4$} (5,0) node[below] {$\sst 3$} (7,2) node[above] {$\sst 5$} (3,0) node[below] {$\sst 4$};
\end{tikzpicture}
\end{minipage}

\caption{Decomposition into two cylinders in $\Hyp$}
\label{fig:2cyl:decomp:models}
\end{figure}

\item[(iv)] If $M$ has only one horizontal cylinder, then the top and bottom boundaries of this cylinder are glued following a unique model given in Figure~\ref{fig:1cyl:decomp:model}.

\begin{figure}[htb]
\centering
\begin{tikzpicture}[scale=0.6]
\draw (0,0) -- (10,0) -- (10,2) -- (0,2) -- cycle;

\foreach \x in {(0,2), (0,0), (2,2), (2,0), (4,2), (4,0), (6,2), (6,0), (8,2), (8,0), (10,2), (10,0)} \filldraw[fill=white] \x circle (2pt);
\draw (1,2) node[above] {$\sst 1$} (3,2) node[above] {$\sst 2$} (5,2) node[above] {$\sst 3$} (7,2) node[above] {$\sst 4$} (9,2) node[above] {$\sst 5$} (1,0) node[below] {$\sst 5$}
(3,0) node[below] {$\sst 4$} (5,0) node[below] {$\sst 3$} (7,0) node[below] {$\sst 2$} (9,0) node[below] {$\sst 1$};

\end{tikzpicture}

\caption{Decomposition into one cylinder in $\Hyp$.}
\label{fig:1cyl:decomp:model}
\end{figure}
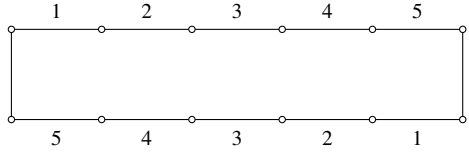
\end{itemize}

\end{Lemma}

\begin{proof}\hfill
\begin{itemize}
\item[(i)] By result of Masur (see~\cite{Mas86}) we know that a surface in $\Hg$ cannot be decomposed into more than $g+|\ul{k}|-1$ cylinders, where $|\ul{k}|$  is the length of the vector $\ul{k}$. In this case we have $g=3, |\ul{k}|=1$, thus there are at most three cylinders.

\item[(ii)] Observe that each saddle connection is contained in the top boundary of a unique cylinder. Therefore, a cylinder decomposition determines a partition $\mathcal P$ of the set of horizontal saddle connections. We will classify the partitions by the numbers of saddle connections in the subsets. If there are three horizontal cylinders, then we have two kinds of partitions:
\begin{itemize}
\item[$\bullet$] $\mathcal{P}=(1,2,2)$: in this case there is a unique simple cylinder which we will denote by $C_1$. By Corollary~\ref{cor:cyl:in:hyp} (b), $C_1$ is adjacent to a unique cylinder $C_2$, we denote the remaining cylinder by $C_3$. It is now easy to check that there is  only one topological model corresponding to this partition which is given on the left of Figure~\ref{fig:3cyl:decomp:models}.

\item[$\bullet$] $\mathcal{P}=(1,1,3)$: in this case, there are two simple cylinders which will be denoted by $C_1,C_2$. Note that $C_1$ and $C_2$ cannot be adjacent, since otherwise the unique saddle connection in the top boundary of $C_1$ and the bottom one of $C_2$ would be a regular simple closed geodesic. It follows that both $C_1$ and $C_2$ are adjacent to the remaining cylinder $C_3$. Again, there is only one topological model for such a partition, which is shown on the right of Figure~\ref{fig:3cyl:decomp:models}.
\end{itemize}

\item[(iii)] Recall that the hyperelliptic involution of $M$ has exactly 8 fixed points, one of which is the unique singularity of $M$. Since each invariant cylinder contains exactly two fixed points in the interior,  three of the fixed points must be contained in the interiors of the horizontal saddle connections. Note that each saddle connection contains at most one fixed point, in which case it is invariant by the hyperelliptic involution. Thus there are three saddle connections that are invariant by  the hyperelliptic involution. By Corollary~\ref{cor:cyl:in:hyp} (a), each of them is contained in the closure of a unique cylinder.

Since there are two cylinders,  the possible partitions are $\mathcal{P}=(2,3)$ and $\mathcal{P}=(1,4)$.

\begin{itemize}
\item[$\bullet$] $\mathcal{P}=(2,3)$: let $C_1$ denote the cylinder whose top boundary contains two saddle connections, and $C_2$ denote the other one.   Observe that the top boundary of a cylinder cannot contain only invariant saddle connections. Therefore, we must have that the top boundary of $C_1$ contains one invariant saddle connection, and the top boundary of $C_2$ contains two invariant saddle connections. The only possible configuration is shown in Figure~\ref{fig:2cyl:decomp:models}, left.

\item[$\bullet$] $\mathcal{P}=(1,4)$: since the unique saddle connection in the top boundary of the simple cylinder cannot be invariant, the three invariant saddle connections must be contained in the top boundary of the other cylinder. Thus, there is only one possible configuration, which is given in Figure~\ref{fig:2cyl:decomp:models}, right.
\end{itemize}

\item[(iv)] Since there is only one cylinder, all the horizontal saddle connections are invariant by the hyperelliptic involution, and there is only one possible configuration, which is shown in Figure~\ref{fig:1cyl:decomp:model}.
\end{itemize}
\end{proof}

\section{Dimension $4$ submanifolds of $\H(4)^{\rm hyp}$}

Since all the period coordinates of $\H^{\rm hyp}(4)$ are absolute periods, an immediate consequence of Theorem~\ref{th:symp:cond} is

\begin{Corollary}\label{cor:symp:cond:H4}
 If $\smN$ is an affine invariant submanifold of $\H^{\rm hyp}(4)$, then $T_\R(\smN)$  is symplectic. In particular, we have
  $$\dim_\C\smN \in \{2,4,6\}.$$
\end{Corollary}
Let $(M,\omega)$ be a surface in  $\Hyp$, and $\smN$ be its $\GL^+(2,\R)$-orbit closure. By Theorem~\ref{th:ob:cl}, $\smN$ is an affine invariant submanifold of $\Hyp$. If $\dim_\C\smN=2$, then $\smN=\GL^+(2,\R)\cdot(M,\omega)$, which means that the $\GL^+(2,\R)$-orbit of $(M,\omega)$ is closed, or equivalently $(M,\omega)$ is a Veech surface. Since $\dim_\C\Hyp=6$, if $\dim_\C\smN=6$, then $\smN=\Hyp$, which means that $\GL^+(2,\R)\cdot(M,\omega)$ is dense in $\Hyp$. Therefore, to prove Theorem~\ref{th:main:Hyp4}, we only need to rule out the case $\dim_\C\smN=4$. We first have

\begin{Proposition}\label{P:3done}
If $\smN$ is an affine invariant submanifold in $\H(4)$ and $\smN$ contains a surface with three free parallel cylinders, then $\smN$ is equal to a connected component of $\H(4)$.
\end{Proposition}

\begin{proof}
Since the three cylinders are free, they can be sheared independently. The derivatives of shears in these three cylinders span a three dimensional Lagrangian in $T_\R(\smN)=p(T_\R(\smN))$. By Theorem \ref{th:symp:cond}, $\smN$ must have dimension at least 6, and this is equal to the dimension of $\H(4)$.
\end{proof}

\begin{Proposition}\label{P:3h}
Every dimension $4$ affine invariant submanifold $\smN$ of $\Hyp$ contains a horizontally periodic translation surface with 3 horizontal cylinders.
\end{Proposition}

\begin{figure}[htb]
\begin{minipage}[t]{0.8\linewidth}
\centering
\begin{tikzpicture}[scale=0.6]
\fill[yellow!70!blue!30] (-7,0) -- (-5,2) -- (-3,2) -- (-5,0) -- cycle;
\fill[pattern = dots, pattern color=black] (-7,0) -- (-5,2) -- (-3,2) -- (-5,0) -- cycle;
\fill[blue!70] (-3,0) -- (-1,0) -- (-1,2) -- (-3,2) -- cycle;
\draw (-9,4) -- (-9,2) -- (-7,2) -- (-7,0) -- (-1,0) -- (-1,2) -- (-5,2) -- (-5,4) -- cycle;
\draw (-7,2) -- (-5,2);

\foreach \x in {(-9,4), (-9,2), (-7,4), (-7,2), (-7,0), (-5,4), (-5,2), (-5,0), (-3,2), (-3,0), (-1,2), (-1,0)} \filldraw[fill=white] \x circle (2pt);

\draw (-8,4) node[above] {$\sst 1$} (-8,2) node[below] {$\sst 1$} (-6,4) node[above] {$\sst 2$} (-4,0) node[below] {$\sst 2$} (-6,2) node[above] {$\sst 3$} (-4,2) node[above] {$\sst 4$} (-6,0) node[below] {$\sst 4$} (-2,2) node[above] {$\sst 5$} (-2,0) node[below] {$\sst 5$};

\fill[green!30!yellow!70] (4,2) -- (4,1.5) -- (8,1.5) -- (6,2);
\fill[green!30!yellow!70] (8,1.5) -- (10,1.5) -- (10,2) -- cycle;
\fill[green!30!yellow!70] (4,0) -- (6,-0.5) -- (8,-0.5) -- (10,0) -- cycle;
\draw (2,4) -- (2,2) -- (4,2) -- (4,0) -- (6,-0.5) -- (8,-0.5) -- (10,0) -- (10,2) -- (8,1.5) -- (6,2) -- (6,4) -- cycle;
\draw (4,2) -- (6,2) (4,1.5) -- (10,1.5) (4,0) -- (10,0);

\foreach \x in {(2,4), (2,2), (4,4), (4,2), (4,0), (6,4), (6,2), (6,-0.5),(8,1.5), (8,-0.5), (10,2), (10,0)} \filldraw[fill=white] \x circle (2pt);

\draw (3,4) node[above] {$\sst 1$} (3,2) node[below] {$\sst 1$} (5,4) node[above] {$\sst 2$} (7,-0.5) node[below] {$\sst 2$} (5,2) node[above] {$\sst 3$} (7,2) node {$\sst 4$} (5,-0.5) node {$\sst 4$} (9,2) node {$\sst 5$} (9,-0.5) node {$\sst 5$};

\end{tikzpicture}
\end{minipage}
\begin{minipage}[t]{0.8\linewidth}
\centering
\begin{tikzpicture}[scale=0.6]
\fill[yellow!70!blue!30] (-9,0) -- (-7,2) -- (-5,2) -- (-7,0) -- cycle;
\fill[pattern = dots, pattern color=black] (-9,0) -- (-7,2) -- (-5,2) -- (-7,0) -- cycle;
\fill[blue!70] (-5,0) -- (-3,2) -- (-1,2) -- (-3,0) -- cycle;
\draw (-9,4) -- (-9,0) -- (-1,0) -- (-1,2) -- (-7,2) -- (-7,4) -- cycle;
\draw (-9,2) -- (-7,2);

\foreach \x in {(-9,4), (-9,2), (-9,0), (-7,4), (-7,2), (-7,0), (-5,2), (-5,0), (-3,2), (-3,0), (-1,2), (-1,0)} \filldraw[fill=white] \x circle (2pt);
\draw (-8,4) node[above] {$\sst 1$} (-6,0) node[below] {$\sst 1$} (-8,2) node[above] {$\sst 2$} (-6,2) node[above] {$\sst 3$} (-8,0) node[below] {$\sst 3$} (-4,2) node[above] {$\sst 4$} (-2,0) node[below] {$\sst 4$} (-2,2) node[above] {$\sst 5$} (-4,0) node[below] {$\sst 5$};

\fill[green!30!yellow!70] (2,2) -- (4,2) -- (6,1.5) -- (2,1.5) -- cycle;
\fill[green!30!yellow!70] (8,1.5) -- (10,2) -- (10,1.5) -- cycle;
\fill[green!30!yellow!70] (2,0) -- (8,0) -- (6,-0.5) -- (4,-0.5) -- cycle;
\draw (2,4) -- (2,0) -- (4,-0.5) -- (6,-0.5) -- (8,0) -- (10,0) -- (10,2) -- (8,1.5) -- (6,1.5) -- (4,2) -- (4,4) -- cycle;
\draw (2,2) -- (4,2) (2,1.5) -- (6,1.5) (8,1.5) -- (10,1.5) (2,0) -- (8,0);

\foreach \x in {(2,4), (2,2), (2,0), (4,4), (4,2), (4,-0.5), (6,1.5), (6,-0.5), (8,1.5), (8,0), (10,2), (10,0)} \filldraw[fill=white] \x circle (2pt);

\draw (3,4) node[above] {$\sst 1$} (5,-0.5) node[below] {$\sst 1$} (3,2) node[above] {$\sst 2$} (5,2) node {$\sst 3$} (3,-0.5) node {$\sst 3$} (7,1.5) node[above] {$\sst 4$} (9,0) node[below] {$\sst 4$} (9,2) node {$\sst 5$} (7,-0.5) node {$\sst 5$};
\end{tikzpicture}
\end{minipage}
\caption{Shearing two simple cylinders included in the larger horizontal one simultaneously so that the core curve remains horizontal. We get automatically a surface which is decomposed into three horizontal cylinders. On the right, the third horizontal cylinder is colored.}
\label{fig:twist:2cyl}
\end{figure}
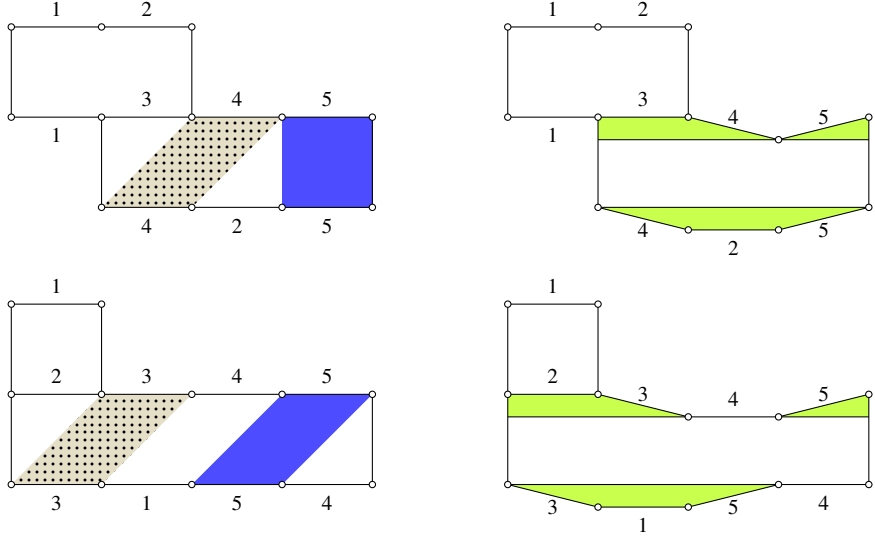

\begin{proof}
Assume in order to find a contradiction that every translation surface in $\smN$ has at most two horizontal cylinders.

Then by Proposition~\ref{P:allfree}, every cylinder on every translation surface $M\in \smN$ is free, which means that we are free to shear and stretch any cylinder while staying in $\smN$. By Proposition \ref{T:manyC}, we can find a surface $M$ which is horizontally periodic with two horizontal cylinders. There are only two  models for such decompositions in $\Hyp$ (see Lemma~\ref{lm:3cyl:models}). Both have a horizontal cylinder $C$ with two transverse simple cylinders contained in the closure of $C$. These can be sheared non-trivially so that the core curve of $C$ remains horizontal. This produces a translation surface $M'$ where the two horizontal cylinders from $M$ persist and remain horizontal. It is not difficult to see that  the remaining part of the surface is also a horizontal cylinder (see Figure~\ref{fig:twist:2cyl}).
\end{proof}

We are now ready to prove the following theorem, which implies immediately Theorem~\ref{th:main:Hyp4}.

\begin{Theorem}\label{T:no:dim4}
There are no $4$ dimensional affine invariant submanifolds of $\Hyp$.
\end{Theorem}

\begin{proof}
Suppose to the contrary that a 4 dimensional affine invariant submanifold $\smN$ of  $\Hyp$ does exist. By Proposition~\ref{P:3h}, it has a surface $(M,\omega)$ which is horizontally periodic with three horizontal cylinders. Let $\alpha_1,\alpha_2,\alpha_3$ denote the core curves of the cylinders. By Theorem~\ref{T:manyC}, we know that $\dim_\C\vect(\alpha_1,\dots,\alpha_3)\leq 2$. If $\dim_\C\vect(\alpha_1,\dots,\alpha_3) < 2$, then there exists a horizontally periodic surface in $\smN$ which has more horizontal cylinders than $M$. But a surface in $\Hyp$ cannot have more than three cylinders, therefore we can conclude that $\dim_\C \vect(\alpha_1,\dots,\alpha_3)=2$.

Thus there exist $r_1,r_2,r_3 \in \R$, not all zero, such that
$$ r_1\alpha_1 + r_2\alpha_2 + r_3\alpha_3 = 0 \text{ in } T(\smN)^*,$$
and this is the only relation between $\alpha_1, \alpha_2, \alpha_3$ in $T(\smN)^*$.

If $r_1r_2r_3 \neq 0$, then no cylinder is $\smN$-parallel to another one. Thus all the cylinders are free, and so Proposition~\ref{P:3done}  implies that $\dim_\C \smN=6$. Since this is a contradiction, we must have $r_1r_2r_3=0$. Since each $\alpha_i\neq 0$ in  $T(\smN)^*$, only one of $\{r_1,r_2,r_3\}$ is zero. The other two are nonzero, which means that two of the cylinders are $\smN$-parallel, and the third one is free.

We only have two topological models for 3-cylinder decompositions in $\Hyp$ (see Lemma~\ref{lm:3cyl:models}). The possible situations are shown in Figure~\ref{fig:3cyl:all:cases}, where the pair of $\smN$-parallel cylinders have been colored. The left and center situations on the top row are actually the same, so there are 5 different situations. We will regroup these situations into three cases. In what follows, we will use  cylinder deformation arguments to get a contradiction in all cases, which shows as desired that we cannot have $\dim_\C\smN=4$.
\begin{figure}[h]
\includegraphics[scale=0.4]{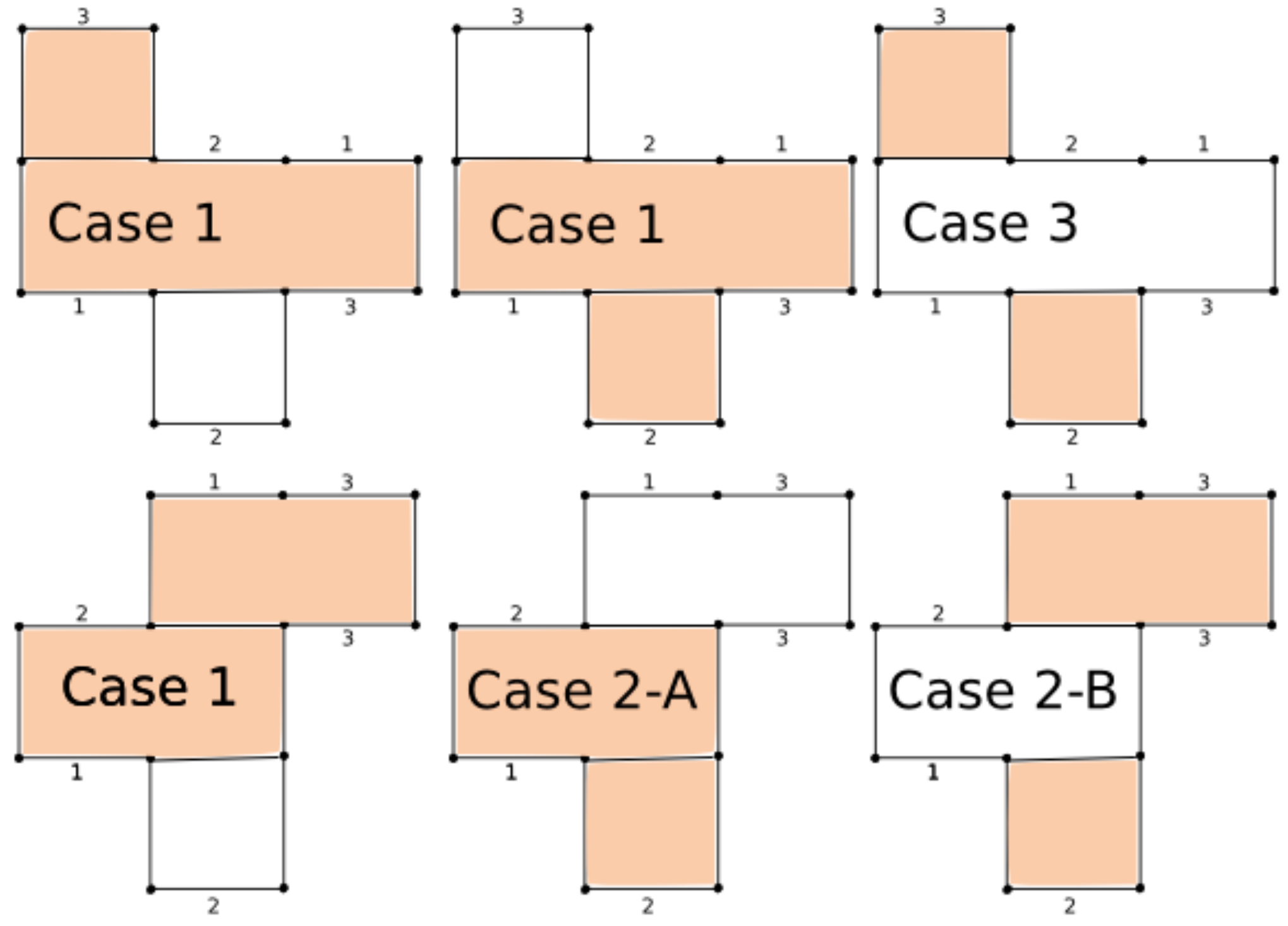}
\caption{To prove the theorem, it suffices to study three cylinder directions, and prove that no horizontal cylinder can be $\smN$-parallel to any other. All of the possible cases for three cylinder directions are listed here, with all possibilities for a pair of $\smN$-parallel horizontal cylinders (colored). The possibilities are grouped into cases specified below.
}
\label{fig:3cyl:all:cases}
\end{figure}

\bold{Case 1.} $M$ has a simple horizontal cylinder $C_1$ that is free. By Corollary~\ref{cor:cyl:in:hyp} (c), $C_1$ is adjacent to a unique horizontal cylinder $C_2$. Let $C_3$ denote the third and final horizontal cylinder. By assumption $C_2$ and $C_3$ are $\smN$-parallel. Shear $C_1$ so that there is a transverse cylinder $D$  whose closure contains $C_1$ and goes through $C_2$ but not $C_3$ (see Figure~\ref{fig:C1:free}).
\begin{figure}[h]
\begin{tikzpicture}[scale=0.5]
\fill[gray!40] (0,6) --(0,4) -- (4,2) -- (6,2) -- (2,4) -- (2,6) -- cycle;

\draw (0,6) -- (0,2) -- (2,2) -- (2,0) -- (4,0) -- (4,2) -- (6,2) -- ( 6,4) -- (2,4) -- (2,6) -- cycle;

\draw (0,4) -- (4,2) (0,5) -- (6,2) (0,6) -- (2,5);

\foreach \x in {(0,6), (2,6), (0,4), (0,2), (2,4), (2,2), (2,0), (4,4), (4,2), (4,0), (6,4), (6,2)} \filldraw[fill=black] \x circle (2pt);

\draw (1,6) node[above] {\tiny $3$} (1,2) node[below] {\tiny $1$} (3,4) node[above] {\tiny $2$} (3,0) node[below] {\tiny $2$} (5,4) node[above] {\tiny $1$} (5,2) node[below] {\tiny $3$};

\draw (3,3) node {\tiny $D$};

\end{tikzpicture}
\caption{An example of Case 1: $C_1$ is  the top horizontal cylinder which is free, the middle horizontal cylinder is $C_2$. After shearing $C_1$, one can find a transverse cylinder $D$ corresponding to the shaded region.}
\label{fig:C1:free}
\end{figure}
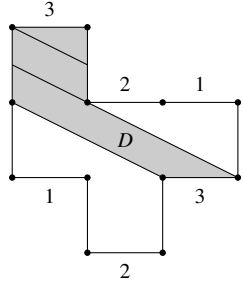

By part (a) of Proposition \ref{P:morefree} (using $C_1$ as the free cylinder), we see that $D$ is free. By part (c) of the same proposition (using $D$ as the free cylinder), we see that $C_2$ is free, which is a contradiction to the hypothesis that $C_2$ and $C_3$ are $\smN$-parallel.

\medskip

\bold{Case 2 (A and B).} $M$ has a simple horizontal cylinder $C_1$ that is $\smN$-parallel to a non-simple horizontal cylinder $C_2$. Let the third horizontal cylinder be $C_3$; this must be free.

\bold{Case 2-A.} $C_1$ is adjacent to $C_2$, and $C_3$ is also adjacent to $C_2$. Note that both $C_2$ and $C_3$ are not simple by assumption, so there is a unique pair of saddle connections that border both $C_2$ and $C_3$.  Since $C_3$ is free, we can shear it to get another surface $(M',\omega')$ in $\smN$ with a transverse cylinder $D$  whose closure contains entirely the horizontal saddle connections between $C_2$ and $C_3$, and intersects only $C_2$ and $C_3$ (see Figure~\ref{fig:C1:no:free:A}).
\begin{figure}[htb]
\centering
\begin{tikzpicture}[scale=0.5]
\fill[yellow] (0,2) -- (2,2) -- (6,6) -- (4,6) -- cycle;
\fill[green!50] (4,4) -- (6,4) -- (8,6) -- (6,6) -- cycle;
\draw (0,4) -- (0,2) -- (2,2) -- (2,0) -- (4,0) -- (4,4) -- (6,4) -- (8,6) -- (4,6) -- (2,4) -- cycle;
\draw (0,2) -- (2,4) (2,2) -- (6,6) (2,2) -- (4,2);

\foreach \x in {(0,4), (0,2), (2,4), (2,2), (2,0), (4,6), (4,4), (4,2) ,(4,0), (6,6), (6,4), (8,6)} \filldraw[fill=white] \x circle (2pt);
\draw (1,2) node[below] {$\sst 1$} (5,6) node[above] {$\sst  1$} (1,4) node[above] {$\sst 2$} (3,0) node[below] {$\sst 2$}  (5,4) node[below] {$\sst 3$} (7,6) node[above] {$\sst 3$};
\draw (3,4) node {$\sst D$} (6,5) node {$\sst D'$} (3,1) node {$\sst C_1$};
\end{tikzpicture}

\caption{Case 2-A: $C_1$ is the bottom horizontal cylinder,  and the free cylinder $C_3$ is the top horizontal cylinder. After shearing $C_3$, we get a transverse cylinder $D$ intersecting $C_2$ and $C_3$, and a parallel cylinder $D'$ contained in the closure of} $C_3$.
\label{fig:C1:no:free:A}
\end{figure}
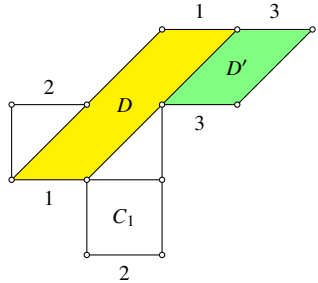

Observe that the complement of $D$ in $C_3$ is another cylinder $D'$ in the same direction of $D$. By  Proposition~\ref{P:morefree} (c) (with free cylinder $C_3$), we see that $D$ is free.

Now Proposition~\ref{P:break} implies that $C_1$ and $C_2$ are not $\smN$-parallel since they do not have an equal portion in $D$, and we get a contradiction.

\medskip
\bold{Case 2-B.} The simple cylinder $C_1$ is adjacent to the free cylinder $C_3$. Note that $C_2$ is also adjacent to $C_3$. The topological model for this case is the same as Case 2-A, we only change the coloring and the numbering of the cylinders. In Figure~\ref{fig:C1:no:free:A}, $C_3$ is the middle cylinder, and $C_2$ is the top one.

By a similar  argument to  the one in Case 2-A, we see that by shearing $C_3$, we can find a surface $(M',\omega')$ in $\smN$ with two parallel transverse cylinders $D$ and $D'$, where $D$ intersects $C_2$ and $C_3$, and $D'$ is contained entirely in $C_2$.

 As in Proposition~\ref{P:morefree} (b), we see that $D'$ is free. (More precisely, we use Proposition~\ref{P:break} after observing that there is no other cylinder parallel to $D'$ which is contained entirely in the equivalence class $\{C_1, C_2\}$.)  By Proposition~\ref{P:morefree} (a) (with free cylinder $D'$), we see that $C_2$ is free, which contradicts to the assumption that $C_1$ and $C_2$ are $\smN$-parallel.

\medskip
\bold{Case 3.} The only remaining case in $\Hyp$ is that there are two simple horizontal cylinders $C_1$ and $C_2$ on $M$ which are $\smN$-parallel, and a third horizontal cylinder $C_3$ which is free but not simple.

\begin{figure}[htb]
\centering
\begin{tikzpicture}[scale=0.7]
\fill[yellow] (-8,4) -- (-6,2) -- (-4,2) -- (-6,4) -- cycle;

\draw (-1,0) -- ((-2,2) -- (-2,4) -- (-4,4) -- (-4,6) -- (-6,6) -- (-6,4) -- (-8,4) -- (-8,2) -- (-4,2) -- (-3,0) -- cycle (-6,2) -- (-8,4) (-4,2) -- (-6,4);
\draw (-6,4) -- (-4,4) (-4,2) -- (-2,2);
\draw[thin, dashed]  (-4,4) -- (-4,2);
\foreach \x in {(-8,4), (-8,2), (-6,6),(-6,4), (-6,2), (-4,6), (-4,4), (-4,2),(-3,0), (-2,4), (-2,2), (-1,0)} \filldraw[fill=white] \x circle (2pt);

\draw (-7,2) node[below] {$\sst 1$} (-5,6) node[above] {$\sst 1$}  (-7,4) node[above] {$\sst 3$} (-5,2) node[below] {$\sst 3$} (-3,4) node[above] {$\sst 2$} (-2,0) node[below] {$\sst 2$};
\draw  (-3,3) node {$\sst C_3$} (-6,3) node {$\sst D$};

\fill[yellow]  (2,4) -- (4,2) -- (4,0) -- (2,2) -- cycle;
\draw (2,0) -- (4,0) -- ( 4,2) -- (5,0) -- (7,0) -- (6,2) -- (6,4) -- (4,4) -- (4,6) -- (2,6) -- cycle (4,0) -- (2,2) (4,2) -- (2,4);
\draw (4,4) -- (4,2);
\foreach \x in {(2,6), (2,4), (2,2),(2,0), (4,6), (4,4), (4,2), (4,0), (5,0), (6,4),(6,2), (7,0)} \filldraw[fill=white] \x circle (2pt);
\draw (3,6) node[above] {$\sst 1$} (3,0) node[below] {$\sst 1$} (5,4) node[above] {$\sst 2$} (6,0) node[below] {$\sst 2$};
\draw (3,5) node {$\sst E$} (3,2) node {$\sst D$};

\end{tikzpicture}
\caption{On the left is the only possible picture in Case 3: the middle cylinder has been sheared so the dotted line is vertical. The horizontal cylinders are $C_1, C_3, C_2$ from top to bottom. Once the lengths of $1$ and $3$ are the same, the left picture can be cut and pasted to get the right picture. Then there is a vertical cylinder $E$, which occupies the left side of the right picture. The cylinder $D$ is in yellow.}
\label{fig:2:sim:cyl:related}
\end{figure}
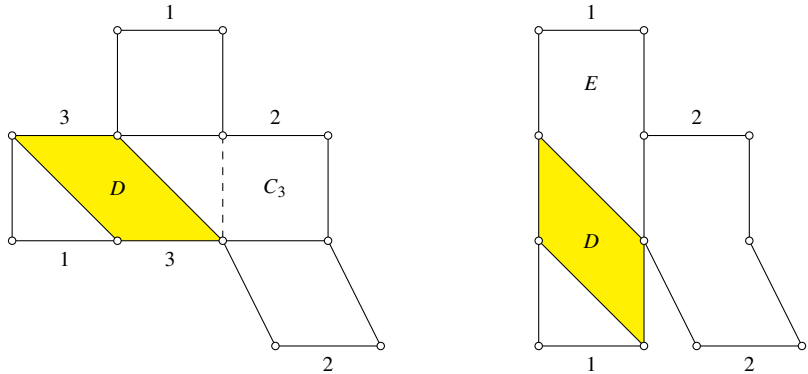

For the arguments in this case, we refer to Figure~\ref{fig:2:sim:cyl:related}. Observe that there is a cylinder $D$ that is contained in the closure of $C_3$, as shown on the left of Figure~\ref{fig:2:sim:cyl:related}. Proposition~\ref{P:morefree} (b) (with free cylinder $C_3$) gives that $D$ is free. Deforming $D$ (stretching it horizontally) we see that the lengths of $1$ and $3$ can be made to be equal. By shearing $C_3$, we get a surface $(M',\omega')$ with a cylinder $E$   whose closure contains both $1$ and $3$ entirely. Note that $C_1$ is included in $E$, and we can suppose that $E$ is vertical.

It follows that $E$ is free by Proposition~\ref{P:morefree} (a) (with free cylinder $D$). Then it follows that $C_1$ is free by  Proposition~\ref{P:morefree} (b) (with free cylinder $E$).
Thus we have a contradiction to the assumption that $C_1$ and $C_2$ are $\smN$-parallel. The proof of Theorem~\ref{T:no:dim4} is now complete.
\end{proof}

\end{document}